\documentclass[oneside]{article}
\usepackage{amsmath,amssymb,amscd,amsthm,verbatim,alltt,amsfonts,array}
\usepackage{mathrsfs}
\usepackage[english]{babel}
\usepackage{latexsym}
\usepackage{amssymb}
\usepackage{euscript}
\usepackage{graphicx}
\usepackage{arydshln}
\usepackage[dvipsnames]{xcolor}

\newtheorem{theorem}{Theorem}
\theoremstyle{plain}

\newtheorem{conjecture}{Conjecture}
\newtheorem{corollary}{Corollary}

\newtheorem{lemma}{Lemma}

\numberwithin{equation}{section}

\begin{document}

{\footnotesize%
\hfill
}

  \vskip 1.2 true cm

\begin{center} {\bf On the limiting extremal vanishing for configuration spaces} \\
          {by}\\
{\sc Muhammad Yameen}
\end{center}

\pagestyle{myheadings}
\markboth{Limiting extremal vanishing for configuration spaces}{Muhammad Yameen}

\begin{abstract}
 We study the limiting behavior of extremal cohomology groups of $k$-points configuration spaces of complex projective spaces of complex dimension $m\geq 4.$ In the previous work, we prove that the extremal cohomology groups of degrees $(2m-2)k+i$ are eventually vanish for each $i\in\{1,2,3\}.$ In this paper, we investigate the extremal cohomology groups for non-positive integers, and show that these cohomology groups are eventually vanish for $i\in\{-1,-2,0\}.$ As an application, we confirm the validity of more general question of Knudsen, Miller and Tosteson for non-positive integers.

We give a certain families of unstable cohomology groups, which are not eventually vanish. The degrees of these families of cohomology groups are depend on the number of points and the dimension of projective spaces.
 
We formulate the conjecture that the cohomology groups of higher slopes are eventually vanish.
\end{abstract}

\begin{quotation}

\end{quotation}

\thispagestyle{empty}

\section{Introduction}

\label{sec:intro}
For any connected manifold $\mathscr{M}$ of finite type (Betti numbers are finite), the space
$$\mathscr{F}_{k}(\mathscr{M}):=\{(x_{1},\ldots,x_{k})\in \mathscr{M}^{k}| x_{i}\neq x_{j}\,for\,i\neq j\}$$
is called the configuration space of $k$ distinct ordered points in $\mathscr{M}.$ The symmetric group $\mathscr{S}_{k}$ acts on $\mathscr{F}_{k}(\mathscr{M})$ by permuting the coordinates. This action is transitive and the orbit $$\mathscr{C}_{k}(\mathscr{M}):=\mathscr{F}_{k}(\mathscr{M})/\mathscr{S}_{k} $$
is the unordered configuration space. Configuration spaces, that are parameter spaces for reduced zero-cycles on manifolds, are cornerstones objects in topology, source of several crucial topological data. For example, when the base manifold if the affine plane, the fundamental groups are the so-called braid groups which are fundamental in geometric group theory. Braid spaces or configuration spaces of unordered pairwise distinct points on manifolds have important applications to a number of areas of mathematics, physics and computer sciences.  

It is a fundamental problem in algebraic topology to understand the homological properties of such spaces. The homological stability of the spaces $\mathscr{C}_{k}(\mathscr{M})$ proved by McDuff \cite{MD}, Segal \cite{S} and Church \cite{C}: for each $i\geq0$ the function $$k\mapsto \text{dim}H_{i}(\mathscr{C}_{k}(\mathscr{M});\mathbb{Q})$$
is eventually constant. This result was extended by Randal-William \cite{RW} and Knudsen \cite{Kn}. 

More recently, Knudsen, Miller and Tosteson \cite{KMT} study the extremal stability of the spaces $\mathscr{C}_{k}(\mathscr{M})$: for each $i\geq0$ the function $$k\mapsto \text{dim}H_{\nu_{k}-i}(\mathscr{C}_{k}(\mathscr{M});\mathbb{Q})$$
is eventually a quasi-polynomial, where $\nu_{k}=(d-1)k+1$ and $dim(\mathscr{M})=d.$ They asked the following question:\\\\
\textbf{Question.} (see Question 4.10 of \cite{KMT}) Suppose that $H_{d-1}(\mathscr{M};\mathbb{Q})=0.$ For $i\in\mathbb{N},$ is the Hilbert function 
$$k\mapsto \text{dim}H_{k(d-2)+i}(\mathscr{C}_{k}(\mathscr{M});\mathbb{Q})$$
eventually a quasi-polynomial?\\\\

For the definition of quasi-polynomial, see section 2.3 of \cite{KMT}. Since the initial draft of the paper of Knudsen, Miller and Tosteson \cite{KMT} appeared, the author of this paper has answered the above question in the affirmative in the case of complex projective spaces:
\begin{theorem}\label{maino}\cite{Y}
For $i,m\in\mathbb{N},$ the Hilbert function 
$$k\mapsto \emph{dim}H_{k(2m-2)+i}(\mathscr{C}_{k}(\emph{CP}^{m});\mathbb{Q})$$
is eventually a quasi-polynomial.
\end{theorem}
Actually, the cohomology groups $$H^{k(2m-2)+i}(\mathscr{C}_{k}(\text{CP}^{m});\mathbb{Q})$$ are eventually vanished for $m>1$ and $i\in\mathbb{N}.$ In particular, we proved that these cohomology groups are entirely vanish for $i>\mu_{k}$ and $k\geq1,$ where $\mu_{k}=(2m-2)k+3.$ We called this vanishing is \emph{entire extremal vanishing}. Moreover, the cohomology groups $H^{k(2m-2)+i}(\mathscr{C}_{k}(\text{CP}^{m});\mathbb{Q})$ are eventually vanish for $i\in\{1,2,3\}.$ For small value of $k,$ these cohomology groups are not necessarily vanish. We called this vanishing is \emph{limiting extremal vanishing}. In the new published form of the paper, Knudsen, Miller and Tosteson \cite{KMTP} extent the above question from natural numbers to integer numbers:\\\\
\textbf{Question.} (see Question 4.11 of \cite{KMTP}) Suppose that $H_{d-1}(\mathscr{M};\mathbb{Q})=0.$ For $i\in\mathbb{Z},$ is the Hilbert function 
$$k\mapsto \text{dim}H_{k(d-2)+i}(\mathscr{C}_{k}(\mathscr{M});\mathbb{Q})$$
eventually a quasi-polynomial?\\\\ 
\textbf{Note:} We will consider the higher dimensional projective spaces $\text{CP}^{m\geq4}.$ The explicit computations for low dimensional projective spaces $\text{CP}^{m<4}$ are already present in the literature (see \cite{F-Ta}, \cite{RW2} and \cite{K-M}). The limiting behavior of the spaces $\mathscr{C}_{k}(\text{CP}^{1})$ and $\mathscr{C}_{k}(\text{CP}^{2})$ are also discussed by Vakil-Wood (see Conjecture H of \cite{VW}).

We investigate the limiting behavior of extremal cohomology groups for non-positive integers. The limiting extremal vanishing can be extent to non-positive value of $i:$
\begin{theorem}\label{main1}
For each $i\in\{-2,-1,0\}$ and $m\geq4,$ the sequence of cohomology groups  
$$\{H^{k(2m-2)+i}(\mathscr{C}_{k}(\emph{CP}^{m});\mathbb{Q})\}_{k=1}^{\infty}$$
is eventually vanish.
\end{theorem}
As an application of Theorem \ref{main1}, we confirm the validity of the question of Knudsen, Miller and Tosteson for non-positive value of $i:$
\begin{corollary}\label{corollarymain}
For $m\geq4$ and $i\in\{-2,-1,0\}$ the Hilbert function 
$$k\mapsto \emph{dim}(H_{k(2m-2)+i}(\mathscr{C}_{k}(\emph{CP}^{m});\mathbb{Q}))$$
is eventually a quasi-polynomial.
\end{corollary}
\begin{center}
\begin{picture}(250,150)
\put(30,20){\vector(0,1){120}} 
\put(20,138){$k$} \put(225,10){$i$}
\put(30,20){\vector(1,0){200}}
\put(30,20){\vector(3,4){90}}
\put(30,20){\vector(2,1){200}}
\put(30,20){\vector(4,1){205}}
\put(237,66){$k=\mu_{k}$}

\put(55,5){$\text{Figure 1. Extremal vanishings}$}
\put(35,110){$\text{Homological}$} \put(43,100){$\text{stability}$}
\put(161,78){$\text{Limiting}$} \put(202,78){$\text{vanishing}$}
\put(135,31){$\text{Entire}$} \put(165,31){$\text{vanishing}$}

\end{picture}
\end{center}

Apart from extremal cohomology groups, we give a certain families of unstable cohomology groups, which are not eventually vanish.
\begin{theorem}\label{main2}
For each $m\geq4$ and $a\in\{2,4,\ldots,2\lceil\frac{m}{2}\rceil-2\},$ we have non-vanishing
$$\displaystyle{\lim_{k \to \infty}}\emph{dim}(H^{a(k-3)+4m-1}(\mathscr{C}_{k}(\emph{CP}^{m});\mathbb{Q}))\neq0.$$
\end{theorem}
It seems that the cohomology groups of higher slopes are eventually vanish. We formulate the following conjecture.
\begin{conjecture}
For each $m\geq4$ and $a\in\{2\lceil\frac{m}{2}\rceil,2\lceil\frac{m}{2}\rceil+2,\ldots,2m-4\},$ we have vanishing
$$\displaystyle{\lim_{k \to \infty}}\emph{dim}(H^{a(k-3)+4m-1}(\mathscr{C}_{k}(\emph{CP}^{m});\mathbb{Q}))=0.$$
\end{conjecture}

\subsection{Outline of the paper and general conventions}
In section 2, we give a quick tour of Chevalley–Eilenberg complex. In section 3, we explicitly discuss the general properties of reduced Chevalley–Eilenberg complex defined by author. The proof of Theorem \ref{main1} is contain in section 4. In section 5, we give the proof of Theorem \ref{main2}. In the last section, we give the final remark on the optimal range of the limiting extremal vanishing.\\\\
$\bullet$ We work throughout with finite dimensional graded vector spaces. The degree of an element $v$ is written $deg(v)$.\\\\
$\bullet$ The symmetric algebra $Sym(\mathscr{V}^{*})$ is the tensor product of a polynomial algebra and an exterior algebra:
$$ Sym(\mathscr{V}^{*})=\bigoplus_{k\geq0}Sym^{k}(\mathscr{V}^{*})=Poly(\mathscr{V}^{even})\bigotimes Ext(\mathscr{V}^{odd}), $$
where $Sym^{k}$ is generated by the monomials of length $k.$\\\\
$\bullet$ Throughout the paper, we will consider the homology and cohomology over $\mathbb{Q}$.\\\\
$\bullet$ The $n$-th suspension of the graded vector space $\mathscr{V}$ is the graded vector space $\mathscr{V}[n]$ with
$\mathscr{V}[n]_{i} = \mathscr{V}_{i-n},$ and the element of $\mathscr{V}[n]$ corresponding to $a\in \mathscr{V}$ is denoted $s^{n}a;$ for example
$$ H_{*}(S^{2};\mathbb{Q})[n] =\begin{cases}
      \mathbb{Q}, & \text{if $*\in\{n,n+2 \}$} \\
      0, & \mbox{otherwise}.\\
   \end{cases} $$ \\\\
$\bullet$ We write $H_{-*}(\mathscr{M};\mathbb{Q})$ for the graded vector space whose degree $-i$ part is
the $i$-th homology group of $M;$ for example
$$ H_{-*}(\text{CP}^m;\mathbb{Q}) =\begin{cases}
      \mathbb{Q}, & \text{if $*\in\{-2m,-2m+2,\ldots,0. \}$} \\
      0, & \mbox{otherwise}.\\
   \end{cases} $$

\section{Chevalley–Eilenberg complex}
Fulton--Macpherson \cite{F-M} described a model $F(k)$ for the cohomology of $\mathscr{F}_{k}(\mathscr{X})$ of a smooth projective variety $\mathscr{X}$, where $F(k)$ depends on the cohomology ring of $\mathscr{X}$, the canonical orientation class and the Chern classes of $\mathscr{X}$. A simplified version of the Fulton--MacPherson model is obtained by Kriz \cite{K} (see also \cite{BMP}). The kriz's model does not depend on Chern classes. The natural action of the symmetric group on the configuration spaces $\mathscr{F}_{k}(\mathscr{X})$ induces an action on the Kriz model. The cohomology of $\mathscr{C}_{k}(\mathscr{X})$ is obtained by the $\mathscr{S}_{k}$--invariant part of Fulton--MacPherson and Kri\v{z} models (see corollary 8c of \cite{F-M} and remark 1.3 of \cite{K}):$$H^{i}(\mathscr{C}_{k}(\mathscr{X});\mathbb{Q})\approx H^{i}(\mathscr{F}_{k}(\mathscr{X});\mathbb{Q})^{\mathscr{S}_{k}}.$$ F\'{e}lix--Thomas \cite{F-Th} (see also \cite{F-Ta}) constructed a Sullivan model for the rational cohomology of configuration spaces of closed oriented even dimensional manifolds. The identification was established in full generality by the Knudsen in \cite{Kn} using the theory of factorization homology \cite{AF}. We will restrict our attention to the case of closed even dimensional manifolds.\\

Let $\mathscr{M}$ be a connected closed oriented manifold of dimension $2m.$ The diagonal comultiplication $\Delta\,:\,H_{*}(\mathscr{M})\rightarrow H_{*}(\mathscr{M})\otimes H_{*}(\mathscr{M})$
is defined on a dual basis $x_{l}^{*}\in H_{*}(\mathscr{M})$ as
$$\Delta(x_{l}^{*})=\sum_{i,j}(\text{coefficient of $x_{l}$ in $x_{i}\cup x_{j}$})x_{i}^{*}\otimes x_{j}^{*},\quad
\text{where $x_{i},\,x_{j}\in H^{*}(\mathscr{M})$}.$$ 

We consider the two shifted copies of vector spaces $$\mathscr{V}^{*}=H_{-*}(\mathscr{M};\mathbb{Q})[2m],\quad\mathscr{W}^{*}=H_{-*}(\mathscr{M};\mathbb{Q})[4m-1]$$
$$ \mathscr{V}^{*}=\bigoplus_{i=0}^{2m}\mathscr{V}^{i},\quad\mathscr{W}^{*}=\bigoplus_{j=2m-1}^{4m-1}\mathscr{W}^{j},$$
and a differential $\partial $ (induced by $\Delta$):
$$\partial|_{\mathscr{V}^{*}}=0,\quad \partial|_{\mathscr{W}^{*}}:\,\mathscr{W}^{*} \longrightarrow Sym^{2}(\mathscr{V}^{*}).$$ We choose bases in $\mathscr{V}^{i}$ and $\mathscr{W}^{j}$ as $$ \mathscr{V}^{i}=\mathbb{Q}\langle v_{i,1},v_{i,2},\ldots\rangle,\quad \mathscr{W}^{j}=\mathbb{Q}\langle w_{j,1},w_{j,2},\ldots\rangle $$
(the degree of an element is marked by the first lower index). Now we consider the graded algebra:

$$ \Omega^{*,*}_{k}(\mathscr{M})=\bigoplus_{i\geq 0}\bigoplus_{\omega=0}^{\left\lfloor\frac{k}{2}\right\rfloor}
\Omega^{i,\omega}_{k}(\mathscr{M})=\bigoplus_{\omega=0}^{\left\lfloor\frac{k}{2}\right\rfloor}\,(Sym^{k-2\omega}(\mathscr{V}^{*})\otimes Sym^{\omega}(\mathscr{W}^{*})) $$
where $i$ and $\omega$ are the total degree and weight grading respectively. 
Now, we get identification $$H^{*}(\mathscr{C}_{k}(\mathscr{M}))\simeq H^{*}(\Omega^{*,*}_{k}(\mathscr{M}),\partial).$$
By definition of differential, we have 
$$\partial:\Omega^{*,*}_{k}(\mathscr{M})\longrightarrow\Omega^{*+1,*-1}_{k}(\mathscr{M}).$$
The relation between configuration spaces of $\mathbb{R}^{m}$ and Lie algebra homology is also studied by Cohen \cite{Co}.

\section{Reduced Chevalley–Eilenberg complex}
In this section, we study the general properties of a reduced complex correspond to $\text{CP}^{m}.$

First, we construct the differential graded algebra $\Omega_{k}^{*,*}(\text{CP}^{m}).$  The cohomology ring of $\text{CP}^{m}$ is:
$$H^{*}(\text{CP}^{m};\mathbb{Q})=\dfrac{\mathbb{Q}[\zeta]}{\langle \zeta^{m+1}\rangle},\quad\text{where  } deg(\zeta)=2.$$
The corresponding two graded vector spaces are
$$\mathscr{V}^{*}=\langle v_{0}, v_{2},\ldots,v_{2m}\rangle,\quad \mathscr{W}^{*}=\langle w_{2m-1}, w_{2m+1},\ldots,w_{4m-1}\rangle.$$
We can write explicit $\partial$ (differential):
$$\partial(v_{2i})=0\qquad \quad\qquad\qquad 0\leq i\leq m,$$
$$\partial(w_{2i-1})=\sum_{\substack{a+b=i \\ 0\leq a, b\leq m}}v_{2a}v_{2b}\qquad m\leq i\leq 2m.$$
We have an isomorphism:
$$H^{*}(\mathscr{C}_{k}(\text{CP}^{m}))\simeq H^{*}(\Omega^{*,*}_{k}(\text{CP}^{m}),\partial).$$

\begin{lemma}\label{lemma1}\cite{Y}
For $k\geq 2,$ the sub-complex $\Omega_{k-2}^{*,*}(\emph{CP}^{m}).(v_{2m}^{2}, w_{4m-1})$ of $\Omega_{k}^{*,*}(\emph{CP}^{m})$ is a-cyclic. 
\end{lemma}
\begin{proof}
An element in $\Omega_{k-2}^{*,*}(\text{CP}^{m}).(v_{2m}^{2}, w_{4m-1})$ has a unique expansion $v_{2m}^{2}\beta+\gamma w_{4m-1},$ where $\beta$ and $\gamma$ have no monomial containing $w_{4m-1}.$ The operator $$h(v_{2m}^{2}\beta+\gamma w_{4m-1})=w_{4m-1}\beta$$ gives a homotopy $id\simeq 0.$
\end{proof}
We denote the reduced complex $(\Omega_{k}^{*,*}(\text{CP}^{m})/\Omega_{k-2}^{*,*}(\text{CP}^{m}).(v_{2m}^{2}, w_{4m-1}),\partial_{\text{induced}})$ by 
$$({}^{r}\Omega_{k}^{*,*}(\text{CP}^{m}),\partial).$$
\begin{corollary}\label{corollary}
For $k\geq 2,$ we have an isomorphism $H^{*}({}^{r}\Omega_{k}^{*,*}(\emph{CP}^{m}),\partial)\cong H^{*}(\mathscr{C}_{k}(\emph{CP}^{m})).$
\end{corollary}
Now, we explicitly discusses the support of the reduced complex.
\begin{lemma}
For $\omega>\emph{min}\{\lfloor\frac{k}{2}\rfloor,m\},$ we have ${}^{r}\Omega_{k}^{\omega,*}(\emph{CP}^{m})=0.$
\end{lemma}
\begin{proof}
The odd degree elements are concentrated in $\mathscr{W}^{*}/\langle w_{4m-1}\rangle.$ The weight grading $\omega$ in ${}^{r}\Omega_{k}^{\omega,*}(\text{CP}^{m})$ is depend on the length of monomials in $Sym(\mathscr{W}^{*}/\langle w_{4m-1}\rangle).$ The monomial of maxiam length in $Sym(\mathscr{W}^{*}/\langle w_{4m-1}\rangle)$ is $w_{2m-1}w_{2m+1}\ldots w_{4m-3}.$ This implies that $\omega\leq m.$ Also, by definition of complex ${}^{r}\Omega_{k}^{*,*}(\text{CP}^{m})$ the wight grading must be less than and equal to $\lfloor\frac{k}{2}\rfloor.$  This complete the proof. 
\end{proof}
Now, we can write the reduced graded algebra as:
$${}^{r}\Omega_{k}^{*,*}(\text{CP}^{m})=\bigoplus_{\omega=0}^{\text{min}\{\lfloor\frac{k}{2}\rfloor,m\}}{}^{r}\Omega_{k}^{\omega,*}(\text{CP}^{m}).$$
Let us consider the set:
$$\Phi_{\omega,k,m}=\{deg(x)\in\mathbb{N}\cup\{0\}\,| \,x\in {}^{r}\Omega_{k}^{\omega,*}(\text{CP}^{m})\}.$$
\begin{lemma}
For $k\geq1,$ $m\geq1$ and $0<\omega\leq \emph{min}\{\lfloor\frac{k}{2}\rfloor,m\}$ we have
$$ \emph{max}\Phi_{\omega,k,m} =\begin{cases}
      4\omega m-(\omega^{2}+\omega), & \text{if $k\geq 2$ even, $\omega=\lfloor\frac{k}{2}\rfloor,$ $\text{min}\{\lfloor\frac{k}{2}\rfloor,m\}=\lfloor\frac{k}{2}\rfloor$} \\
      (2m-2)k-(\omega^{2}-2\omega-2), & \mbox{otherwise}.\\
   \end{cases} $$
\end{lemma}
\begin{proof}
For $\omega=0,$ the highest degree monomial in reduced complex is $v_{2m-2}^{k-1}v_{2m}.$ The degree of this monomial is $(2m-2)k+2.$ 
Let $\omega=\lfloor\frac{k}{2}\rfloor\geq1$ and $\text{min}\{\lfloor\frac{k}{2}\rfloor,m\}=\lfloor\frac{k}{2}\rfloor.$ In this case, the highest degree monomial in the reduced complex ${}^{r}\Omega_{k}^{\omega,*}(\text{CP}^{m})$ is $$w=w_{4m-(2\omega+1)}w_{4m-(2\omega-1)}\ldots w_{4m-3}.$$ The degree of this monomial is:
\begin{align*}
\text{deg}(w)=&4m-(2\omega+1)+4m-(2\omega-1)+\ldots+4m-5+4m-3\\
=&\underbrace{4m+4m+\ldots+4m}_{\omega-\text{times}}-(3+5+\ldots+(2\omega+1))\\
=&4\omega m-(\omega^{2}+2\omega).
\end{align*}
Let $\omega\geq 1.$ Suppose either $\omega\neq\lfloor\frac{k}{2}\rfloor\geq1$ or $\text{min}\{\lfloor\frac{k}{2}\rfloor,m\}\neq\lfloor\frac{k}{2}\rfloor.$ The highest degree monomial in this case is $$u=v_{2m-2}^{k-2\omega-1}v_{2m}w_{4m-(2m+1)}\ldots w_{4m-3}.$$
The degree of this monomial is:
\begin{align*}
\text{deg}(u)=&(2m-2)(k-2\omega-1)+2m+\{4m-(2m+1)+\ldots+4m-3\}\\
=&(2m-2)k-4\omega m+2+\underbrace{4m+4m+\ldots+4m}_{\omega-\text{times}}-(3+5+\ldots+(2\omega+1))\\
=&(2m-2)k-4\omega m+2+4\omega m-(\omega^{2}+2\omega)\\
=&(2m-2)k-(\omega^{2}-2\omega-2).
\end{align*}
\end{proof}

\begin{lemma}
For $k\geq1,$ $m\geq1$ and $0<\omega\leq \emph{min}\{\lfloor\frac{k}{2}\rfloor,m\}$ we have
$$\emph{min}\Phi_{\omega,k,m} = 2\omega m+\omega(\omega-2).$$
\end{lemma}
\begin{proof}
For $\omega=0,$ the lowest degree monomial in reduced complex is $v_{0}^{k}.$ Let $\omega>0.$ The lowest degree monomial in this case is $$v=v_{0}^{k-2\omega}w_{2m-1}\ldots w_{2m+2\omega-3}.$$
The degree of this monomial is:
\begin{align*}
\text{deg}(v)=&(2m-1)+\ldots+(2m+2\omega-3)\\
=&\underbrace{2m+2m+\ldots+2m}_{\omega-\text{times}}+\{-1+1+\ldots+(2\omega-3)\}\\
=&2\omega m+\omega^{2}-2\omega.
\end{align*}

\end{proof}

\begin{center}
\begin{picture}(360,180)

\put(10,40){\vector(0,1){130}} 
\put(10,40){\vector(1,0){310}} 
\put(7,173){$\omega$} \put(322,35){$i$}
\multiput(8,38)(20,0){15}{$\bullet$} \multiput(80,58)(20,0){12}{$\bullet$} 
\multiput(128,78)(20,0){9}{$\bullet$} \multiput(160,98)(20,0){7}{$\bullet$}
\multiput(190,118)(20,0){4}{$\bullet$}
\multiput(20,30)(20,0){13}{$\ldots$} \multiput(210,130)(20,0){2}{$\vdots$} 
\put(3,38){$0$} \put(3,58){$1$} \put(3,78){$2$} \put(3,98){$3$} \put(3,118){$4$}
\put(4,130){$\vdots$}
\put(290.5,39){\line(0,1){10}}\put(290.5,49){\vector(1,0){21}}\put(312,45){$(2m-2)k+2$}
\put(302,61){\vector(1,0){9}}\put(312,58){$(2m-2)k+3$}
\put(290.5,80){\vector(1,0){20}}\put(312,77){$(2m-2)k+2$}
\put(282,100.5){\vector(1,0){30}}\put(312,97){$(2m-2)k-1$}
\put(255,120.5){\vector(1,0){57}}\put(312,117){$(2m-2)k-6$}

\put(80,60.5){\vector(-1,0){10}}\put(40,58){$2m-1$}
\put(128,80.5){\vector(-1,0){10}}\put(103,78){$4m$}
\put(160,100.5){\vector(-1,0){10}}\put(118,98){$6m+3$}
\put(190,120.5){\vector(-1,0){10}}\put(148,118){$8m+8$}
\put(120,10){$\text{Figure 2. Reduced complex}$}
\end{picture}
\end{center}

\begin{lemma}
The differential $\partial$ on the right side of the complex ${}^{r}\Omega_{k}^{*,*}(\emph{CP}^{m})$ is injective except the following cases\\
(i) $\omega=0$\\
(ii) $\omega=1$ and $k\geq3.$
\end{lemma}
\begin{proof}
For each weight $\omega\geq0,$ the right side of reduced complex is generated by single element. For $\omega=0,$ the differential of monomial $v_{2m-2}^{k-1}v_{2m}$ is zero. Also, for $\omega=3$ and $k\geq 3,$ we have $\partial(v_{2m-2}^{k-3}v_{2m}w_{4m-3})=0.$ Apart from these cases, the differential is non-trivial on the right side of the reduced complex.
\end{proof}

\section{Proof of Theorem \ref{main1}}

In this section, we give the proof of Theorem \ref{main1}. The matrix of differential $$\partial:\Omega^{a,b}_{k}(\text{CP}^{m})\longrightarrow\Omega^{a+1,b-1}_{k}(\text{CP}^{m})$$ is denoted by $M_{a,b}.$
\begin{lemma}\label{lemma1}
For each $m\geq4$ and $k>8,$ we have exact sequence (in the square brackets are given the dimensions)
$$0\longrightarrow\underset{[2]}{{}^{r}\Omega_{k}^{(2m-2)k-3,3}(\emph{CP}^{m})}\longrightarrow\underset{[6]}{{}^{r}\Omega_{k}^{(2m-2)k-2,2}(\emph{CP}^{m})}\longrightarrow \underset{[6]}{{}^{r}\Omega_{k}^{(2m-2)k-1,1}(\emph{CP}^{m})}\longrightarrow$$ $$\longrightarrow \underset{[2]}{{}^{r}\Omega_{k}^{(2m-2)k,0}(\emph{CP}^{m})}\longrightarrow 0.$$
\end{lemma}
\begin{proof} First we define the bases elements:
\begin{align*}
{}^{r}\Omega_{k}^{(2m-2)k-3,3}(\text{CP}^{m})=&\langle v_{2m-4}v_{2m-2}^{k-8}v_{2m}w_{4m-7}w_{4m-5}w_{4m-3},\,v_{2m-2}^{k-7}v_{2m}w_{4m-9}w_{4m-5}w_{4m-3}\rangle,\\
{}^{r}\Omega_{k}^{(2m-2)k-2,2}(\text{CP}^{m})=&\langle v_{2m-2}^{k-5}v_{2m}w_{4m-7}w_{4m-5},\,v_{2m-4}v_{2m-2}^{k-6}v_{2m}w_{4m-7}w_{4m-3},\\
&v_{2m-2}^{k-4}w_{4m-7}w_{4m-3},\,v_{2m-2}^{k-5}v_{2m}w_{4m-9}w_{4m-3},\\
&v_{2m-6}v_{2m-2}^{k-6}v_{2m}w_{4m-5}w_{4m-3},\,v_{2m-4}^{2}v_{2m-2}^{k-7}v_{2m}w_{4m-5}w_{4m-3} \rangle,\\
{}^{r}\Omega_{k}^{(2m-2)k-1,1}(\text{CP}^{m})=&\langle v_{2m-4}v_{2m-2}^{k-4}v_{2m}w_{4m-5},\,v_{2m-6}v_{2m-2}^{k-4}v_{2m}w_{4m-3},\, v_{2m-2}^{k-3}v_{2m}w_{4m-7},\\ 
&v_{2m-2}^{k-2}w_{4m-5},\,
 v_{2m-4}v_{2m-2}^{k-3}w_{4m-3},\, v_{2m-4}^{2}v_{2m-2}^{k-5}v_{2m}w_{4m-3}\rangle,\\
{}^{r}\Omega_{k}^{(2m-2)k,0}(\text{CP}^{m})=&\langle v_{2m-4}v_{2m-2}^{k-2}v_{2m},\,v_{2m-2}^{k} \rangle.
\end{align*}
The differential $\partial$ is defined on the bases elements as:
\begin{align*}
\partial(v_{2m-4}v_{2m-2}^{k-8}v_{2m}w_{4m-7}w_{4m-5}w_{4m-3})=&2v_{2m-4}^{2}v_{2m-2}^{k-7}v_{2m}w_{4m-5}w_{4m-3}-\\
-&v_{2m-4}v_{2m-2}^{k-6}v_{2m}w_{4m-7}w_{4m-3},\\
\partial(v_{2m-2}^{k-7}v_{2m}w_{4m-9}w_{4m-5}w_{4m-3})=&2v_{2m-6}^{2}v_{2m-2}^{k-6}v_{2m}w_{4m-5}w_{4m-3}-\\
-&v_{2m-2}^{k-5}v_{2m}w_{4m-9}w_{4m-3},\\
\partial(v_{2m-2}^{k-5}v_{2m}w_{4m-7}w_{4m-5})=&2v_{2m-4}v_{2m-2}^{k-4}v_{2m}w_{4m-5}-v_{2m-2}^{k-3}v_{2m}w_{4m-7},\\
\partial(v_{2m-4}v_{2m-2}^{k-6}v_{2m}w_{4m-7}w_{4m-3})=&2v_{2m-4}^{2}v_{2m-2}^{k-5}v_{2m}w_{4m-3},\\
\partial(v_{2m-2}^{k-4}w_{4m-7}w_{4m-3})=&2v_{2m-6}v_{2m-2}^{k-4}v_{2m}w_{4m-3}+2v_{2m-4}v_{2m-2}^{k-3}w_{4m-3}-\\
-&2v_{2m-2}^{k-3}v_{2m}w_{4m-7},\\
\partial(v_{2m-2}^{k-5}v_{2m}w_{4m-9}w_{4m-3})=&2v_{2m-6}v_{2m-2}^{k-4}v_{2m}w_{4m-3}+v_{2m-4}^{2}v_{2m-2}^{k-5}v_{2m}w_{4m-3},\\
\partial(v_{2m-6}v_{2m-2}^{k-6}v_{2m}w_{4m-5}w_{4m-3})=&v_{2m-6}v_{2m-2}^{k-4}v_{2m}w_{4m-3},\\
\partial(v_{2m-4}^{2}v_{2m-2}^{k-7}v_{2m}w_{4m-5}w_{4m-3})=&2v_{2m-4}^{2}v_{2m-2}^{k-5}v_{2m}w_{4m-3},\\
\partial(v_{2m-4}v_{2m-2}^{k-4}v_{2m}w_{4m-5})=&v_{2m-4}v_{2m-2}^{k-2}v_{2m},\\
\partial(v_{2m-6}v_{2m-2}^{k-4}v_{2m}w_{4m-3})=&0,\\
\partial(v_{2m-2}^{k-3}v_{2m}w_{4m-7})=&2v_{2m-4}v_{2m-2}^{k-2}v_{2m},\\
\partial(v_{2m-2}^{k-2}w_{4m-5})=&2v_{2m-4}^{2}v_{2m-2}^{k-2}v_{2m}+v_{2m-2}^{k},\\
\partial(v_{2m-4}v_{2m-2}^{k-3}w_{4m-3})=&v_{2m-4}v_{2m-2}^{k-2}v_{2m},\\
\partial(v_{2m-4}^{2}v_{2m-2}^{k-5}v_{2m}w_{4m-3})=&0,\\
\partial(v_{2m-4}v_{2m-2}^{k-2}v_{2m})=&0,\\
\partial(v_{2m-2}^{k} )=&0.
\end{align*}
Note that the monomial $v_{2m}^{a}$ is zero in ${}^{r}\Omega_{k}^{*,*}(\text{CP}^{m})$ for $a>1.$ The matrices of the differentials are following:
\begin{equation*}
M_{(2m-2)k-3,3}=
\begin{pmatrix}
0 &-1 &0 &0 &0 &2\\
0&0 &0&-1 &2&0
\end{pmatrix},
\quad
M_{(2m-2)k-2,2}=
\begin{pmatrix}
\hfil 2&0&\ \llap{-}1&0 &\hfil 0 &\hfil 0\\
\hfil 0&0 &\hfil 0&0&\hfil 0&\hfil 2\\
\hfil 0&2 &\ \llap{-}2 &0 &\hfil 2&\hfil 0 \\
\hfil 0&2 &\hfil 0 &0 &0 &\hfil 1\\
\hfil 0& 1&\hfil 0& 0&\hfil 0&\hfil 0\\
\hfil 0&0 &\hfil 0&0 &\hfil 0&\hfil 1

\end{pmatrix}
\end{equation*}
\begin{equation*}
M_{(2m-2)k-1,1}=
\begin{pmatrix}
\hfil 1&\hfil 0\\
\hfil 0&\hfil 0\\
\hfil 2&\hfil 0\\
\hfil 2&\hfil 1\\
\hfil 1&\hfil 0\\
\hfil 0&\hfil 0

\end{pmatrix}.
\end{equation*}
We see that the differential $\partial:{}^{r}\Omega_{k}^{(2m-2)k-3,3}(\text{CP}^{m})\rightarrow {}^{r}\Omega_{k}^{(2m-2)k-2,2}(\text{CP}^{m})$ is injective. Also, the differential $\partial:{}^{r}\Omega_{k}^{(2m-2)k-1,1}(\text{CP}^{m})\rightarrow {}^{r}\Omega_{k}^{(2m-2)k,0}(\text{CP}^{m})$ is surjective. The dimensions of the kernal and the image of the map $\partial:{}^{r}\Omega_{k}^{(2m-2)k-2,2}(\text{CP}^{m})\rightarrow {}^{r}\Omega_{k}^{(2m-2)k-1,1}(\text{CP}^{m})$ are respectively 2 and 4. From these computations, we conclude that the sub-complex:
$$0\longrightarrow\underset{[2]}{{}^{r}\Omega_{k}^{(2m-2)k-3,3}(\text{CP}^{m})}\longrightarrow\underset{[6]}{{}^{r}\Omega_{k}^{(2m-2)k-2,2}(\text{CP}^{m})}\longrightarrow \underset{[6]}{{}^{r}\Omega_{k}^{(2m-2)k-1,1}(\text{CP}^{m})}\longrightarrow$$ $$\longrightarrow \underset{[2]}{{}^{r}\Omega_{k}^{(2m-2)k,0}(\text{CP}^{m})}\longrightarrow 0.$$
is exact. 

\end{proof}
\begin{lemma} \label{lemma2}
For each $m\geq4$ and $k>8,$ we have exact sequence 
$$0\longrightarrow\underset{[1]}{{}^{r}\Omega_{k}^{(2m-2)k-1,3}(\emph{CP}^{m})}\longrightarrow\underset{[3]}{{}^{r}\Omega_{k}^{(2m-2)k,2}(\emph{CP}^{m})}\longrightarrow \underset{[3]}{{}^{r}\Omega_{k}^{(2m-2)k+1,1}(\emph{CP}^{m})}\longrightarrow$$ $$\longrightarrow \underset{[1]}{{}^{r}\Omega_{k}^{(2m-2)k+2,0}(\emph{CP}^{m})}\longrightarrow 0.$$
\end{lemma}
\begin{proof} We define the bases elements:
\begin{align*}
{}^{r}\Omega_{k}^{(2m-2)k-1,3}(\text{CP}^{m})=&\langle v_{2m-2}^{k-7}v_{2m}w_{4m-7}w_{4m-5}w_{4m-3}\rangle,\\
{}^{r}\Omega_{k}^{k(2m-2),2}(\text{CP}^{m})=&\langle v_{2m-2}^{k-5}v_{2m}w_{4m-7}w_{4m-3}, v_{2m-4}v_{2m-2}^{k-6}v_{2m}w_{4m-5}w_{4m-3},\\
& v_{2m-2}^{k-4}w_{4m-5}w_{4m-3}\rangle,\\
{}^{r}\Omega_{k}^{k(2m-2)+1,1}(\text{CP}^{m})=&\langle v_{2m-4}v_{2m-2}^{k-4}v_{2m}w_{4m-3}, v_{2m-2}^{k-3}v_{2m}w_{4m-5}, v_{2m-2}^{k-2}w_{4m-3}\rangle,\\
{}^{r}\Omega_{k}^{(2m-2)k+2,0}(\text{CP}^{m})=&\langle v_{2m-2}^{k-1}v_{2m}\rangle.
\end{align*}
The differential $\partial$ is defined on the bases elements as:
\begin{align*}
\partial(v_{2m-2}^{k-7}v_{2m}w_{4m-7}w_{4m-5}w_{4m-3})=&2v_{2m-4}^{2}v_{2m-2}^{k-6}v_{2m}w_{4m-5}w_{4m-3}-\\
-&v_{2m-2}^{k-5}v_{2m}w_{4m-7}w_{4m-3},\\
\partial(v_{2m-2}^{k-5}v_{2m}w_{4m-7}w_{4m-3})=&2v_{2m-4}v_{2m-2}^{k-4}v_{2m}w_{4m-3},\\
\partial(v_{2m-4}v_{2m-2}^{k-6}v_{2m}w_{4m-5}w_{4m-3})=&v_{2m-4}v_{2m-2}^{k-4}v_{2m}w_{4m-3},\\
\partial(v_{2m-2}^{k-4}w_{4m-5}w_{4m-3})=&2v_{2m-4}v_{2m-2}^{k-3}v_{2m}w_{4m-3}+v_{2m-2}^{k-2}w_{4m-3}-\\
-&2v_{2m-2}^{k-3}v_{2m}w_{4m-5},\\
\partial(v_{2m-4}v_{2m-2}^{k-4}v_{2m}w_{4m-3})=&0,\\
\partial(v_{2m-2}^{k-3}v_{2m}w_{4m-5})=&v_{2m-2}^{k-1}v_{2m},\\
\partial(v_{2m-2}^{k-2}w_{4m-3})=&2v_{2m-2}^{k-1}v_{2m},\\
\partial(v_{2m-2}^{k-1}v_{2m})=&0.
\end{align*}

The matrices of the differentials are following:
\begin{equation*}
M_{(2m-2)k-1,3}=
\begin{pmatrix}
\ \llap{-}1&\hfil  2&\hfil 0

\end{pmatrix},
\quad
M_{(2m-2)k,2}=
\begin{pmatrix}
\hfil 2&0&\hfil 0 \\
\hfil 1&0 &0 \\
\hfil 2&\ \llap{-}2&\hfil 1

\end{pmatrix},
\quad
M_{(2m-2)k+1,1}=
\begin{pmatrix}
0 \\
1\\
2
\end{pmatrix}.
\end{equation*}
We see that the differential $\partial:{}^{r}\Omega_{k}^{(2m-2)k-1,3}(\text{CP}^{m})\rightarrow {}^{r}\Omega_{k}^{(2m-2)k,2}(\text{CP}^{m})$ is injective. Also, the differential $\partial:{}^{r}\Omega_{k}^{(2m-2)k+1,1}(\text{CP}^{m})\rightarrow {}^{r}\Omega_{k}^{(2m-2)k+2,0}(\text{CP}^{m})$ is surjective. The dimensions of the kernal and the image of the map $\partial:{}^{r}\Omega_{k}^{(2m-2)k,2}(\text{CP}^{m})\rightarrow {}^{r}\Omega_{k}^{(2m-2)k+1,1}(\text{CP}^{m})$ are respectively 1 and 2. From these computations, we conclude that the sub-complex:
$$0\longrightarrow\underset{[1]}{{}^{r}\Omega_{k}^{(2m-2)k-1,3}(\text{CP}^{m})}\longrightarrow\underset{[3]}{{}^{r}\Omega_{k}^{(2m-2)k,2}(\text{CP}^{m})}\longrightarrow \underset{[3]}{{}^{r}\Omega_{k}^{(2m-2)k+1,1}(\text{CP}^{m})}\longrightarrow$$ $$\longrightarrow \underset{[1]}{{}^{r}\Omega_{k}^{(2m-2)k+2,0}(\text{CP}^{m})}\longrightarrow 0.$$
is exact. 
\end{proof}
\begin{lemma}\label{lemma3}
For each $m\geq4$ and $k>8,$ the cohomology group of degree $(2m-2)k-2$ is vanish in the subcomplex
$$\ldots\longrightarrow\underset{[8]}{{}^{r}\Omega_{k}^{(2m-2)k-3,1}(\emph{CP}^{m})}\longrightarrow \underset{[3]}{{}^{r}\Omega_{k}^{(2m-2)k-2,0}(\emph{CP}^{m})}\longrightarrow 0.$$
\end{lemma}
\begin{proof}
We define the bases elements:
\begin{align*}
{}^{r}\Omega_{k}^{(2m-2)k-3,1}(\text{CP}^{m})=&\langle v_{2m-8}v_{2m-2}^{k-4}v_{2m}w_{4m-3},\,v_{2m-6}v_{2m-2}^{k-4}v_{2m}w_{4m-5},\,v_{2m-2}^{k-3}v_{2m}w_{4m-9},\\
&v_{2m-4}v_{2m-2}^{k-4}v_{2m}w_{4m-7},\,v_{2m-2}^{k-2}w_{4m-7},\,v_{2m-4}v_{2m-2}^{k-3}w_{4m-5},\\
&v_{2m-4}^{2}v_{2m-2}^{k-5}v_{2m}w_{4m-5},\,v_{2m-6}v_{2m-4}v_{2m-2}^{k-5}v_{2m}w_{4m-3}\rangle,\\
{}^{r}\Omega_{k}^{(2m-2)k-2,0}(\text{CP}^{m})=&\langle v_{2m-6}v_{2m-2}^{k-2}v_{2m},\, v_{2m-4}^{2}v_{2m-2}^{k-3}v_{2m},\,v_{2m-4}v_{2m-2}^{k-1}\rangle.
\end{align*}
The differential $\partial$ is defined on the bases elements as:
\begin{align*}
\partial(v_{2m-8}v_{2m-2}^{k-4}v_{2m}w_{4m-3})=&0,\\
\partial(v_{2m-6}v_{2m-2}^{k-4}v_{2m}w_{4m-5})=&v_{2m-6}v_{2m-2}^{k-2}v_{2m},\\
\partial(v_{2m-2}^{k-3}v_{2m}w_{4m-9})=&2v_{2m-6}v_{2m-2}^{k-2}v_{2m}+v_{2m-4}^{2}v_{2m-2}^{k-3}v_{2m},\\
\partial(v_{2m-4}v_{2m-2}^{k-4}v_{2m}w_{4m-7})=&2v_{2m-4}^{2}v_{2m-2}^{k-3}v_{2m},\\
\partial(v_{2m-2}^{k-2}w_{4m-7})=&2v_{2m-6}v_{2m-2}^{k-2}v_{2m}+2v_{2m-4}v_{2m-2}^{k-1},\\
\partial(v_{2m-4}v_{2m-2}^{k-3}w_{4m-5})=&2v_{2m-4}^{2}v_{2m-2}^{k-3}v_{2m}+v_{2m-4}v_{2m-2}^{k-1},\\
\partial(v_{2m-4}^{2}v_{2m-2}^{k-5}v_{2m}w_{4m-5})=&v_{2m-4}^{2}v_{2m-2}^{k-3}v_{2m},\\
\partial(v_{2m-6}v_{2m-4}v_{2m-2}^{k-5}v_{2m}w_{4m-3})=&0,\\
\partial(v_{2m-6}v_{2m-2}^{k-2}v_{2m})=&0,\\
\partial(v_{2m-4}^{2}v_{2m-2}^{k-3}v_{2m})=&0,\\
\partial(v_{2m-4}v_{2m-2}^{k-1})=&0.
\end{align*}
The matrix of the differential is following:
\begin{equation*}
M_{(2m-2)k-3,1}=
\begin{pmatrix}
\hfil 0&0&\hfil 0 \\
1&0 &0 \\
\hfil 2&1&\hfil 0 \\
0&2 &0 \\
\hfil 2&0&\hfil 2 \\
0&2 &1 \\
0&1 &0 \\
\hfil 0&0 &\hfil 0

\end{pmatrix}
\end{equation*}
We see that the differential $\partial:{}^{r}\Omega_{k}^{(2m-2)k-3,1}(\text{CP}^{m})\rightarrow {}^{r}\Omega_{k}^{(2m-2)k-2,0}(\text{CP}^{m})$ is surjective. Hence the cohomology group of degree $(2m-2)k-2$ is vanish in sub-complex:
$$\ldots\longrightarrow\underset{[8]}{{}^{r}\Omega_{k}^{(2m-2)k-3,1}(\text{CP}^{m})}\longrightarrow \underset{[3]}{{}^{r}\Omega_{k}^{(2m-2)k-2,0}(\text{CP}^{m})}\longrightarrow 0.$$

\end{proof}
\textit{Proof of Theorem \ref{main1}.} Let $m\geq 4$ and $k$ is sufficiently large bigger than 8. There is no element of degree higher than $k(2m-2)+3$ in reduced complex $({}^{r}\Omega_{k}^{*,*}(\text{CP}^{m}),\partial).$ Also, the highest degree element of wight $\omega\geq4$ in $({}^{r}\Omega_{k}^{*,*}(\text{CP}^{m}),\partial)$ is $v_{2m-2}^{k-9}v_{2m}w_{4m-9}w_{4m-7}w_{4m-5}w_{4m-3}.$ The degree of monomial $v_{2m-2}^{k-9}v_{2m}w_{4m-9}w_{4m-7}w_{4m-5}w_{4m-3}$ is $(2m-2)k-6.$ Therefore, we just focus on the weights $\omega<4.$ The elements of degrees $k(2m-2),$ $k(2m-2)-1,$ $k(2m-2)-2$ and $k(2m-2)-3$ are concentrated in weights $\omega\leq3.$ The complete proof follows from Lemmas \ref{lemma1}, \ref{lemma2} and \ref{lemma3}.\\\\
\textit{Proof of Corollary \ref{corollarymain}.} It is obvious from Theorem \ref{main1}.

 $\hfill \square$\\
\section{Proof of Theorem \ref{main2}}
In this section, we give the proof of Theorem \ref{main2}.\\\\
\textit{Proof of Theorem \ref{main2}.} Let $k=3$ and $m\geq4.$ We take $$\alpha=\sum_{j=1}^{m}(2m-3j)v_{2j}w_{4m-(2j+1)}\in {}^{r}\Omega_{k}^{4m-1,1}(\text{CP}^{m}).$$
By using the linearity property of differential, we get
$$\partial(\alpha)=\sum_{j=1}^{m}(2m-3j)\partial(v_{2j}w_{4m-(2j+1)}).$$
Also, by using the definition and the leibniz rule of differential, we have 
$$\partial(\alpha)=\sum_{j=1}^{m}[(2m-3j)v_{2j}(\sum_{\substack{a+b=m-j \\ 0\leq a, b\leq m}}v_{2a}v_{2b})]=0.$$
This shows that $\alpha$ is cocycle. For each $a\in\{2,4,\ldots,2m\},$ we take
$$\beta_{a,k}=v_{a}^{k-3}\alpha\in {}^{r}\Omega_{k}^{(k-3)a+4m-1,1}(\text{CP}^{m}).$$
Clearly, $\beta_{a,k}$ is also cocycle for $k\geq3.$ Now, we want to show that $\beta_{a,k}$ is not coboundary for each $a\in\{2,4,\ldots,2\lceil\frac{m}{2}\rceil-2\}.$ We can write $$\beta_{a,k}=\left(2m-\frac{3a}{2}\right)v_{a}^{k-2}w_{4m-(a+1)}+\sum_{\substack{j=1 \\ j\neq\frac{a}{2}}}^{m}(2m-3j)v_{a}^{k-3}v_{2j}w_{4m-(2j+1)}.$$
For $k\geq5,$ the general element of bases in weight 2 is written as $$\gamma=v_{i_{1}}\ldots v_{i_{k-4}}w_{j}w_{l},$$
where $i_{1},\ldots,i_{k-4}\in \{0,2,\ldots,2m\}$ and $j,l\in\{2m-1,\ldots,4m-3\}.$ For each $a\in\{2,4,\ldots,2\lceil\frac{m}{2}\rceil-2\},$ the power of $v_{a}$ in the image of $w_{j}$ and $w_{l}$ is at most one. Therefore, the power of  $v_{a}$ in the image of $\gamma$ is at most $k-3.$ This implies that $\beta_{a,k}$ is a permanent cocycle and never coboundary. Hence, for each $m\geq4,$ $k\geq5$ and $a\in\{2,4,\ldots,2\lceil\frac{m}{2}\rceil-2\},$ we have non-vanishing
$$H^{a(k-3)+4m-1}(\mathscr{C}_{k}(\text{CP}^{m});\mathbb{Q})\neq0.$$
\section{Final remark}
With more careful analysis, one can improve the limiting extremal vanishing range. However, the optimal range is unclear, and we ask the following.\\\\
\textbf{Question.} What is the smallest value of $i\in\mathbb{Z}$ such that
$$\displaystyle{\lim_{k \to \infty}}\text{dim}(H^{(2m-2)k+i}(\mathscr{C}_{k}(\text{CP}^{m});\mathbb{Q}))\neq0?$$

\noindent\textbf{Acknowledgement}\textit{.} The author gratefully acknowledge the support from the ASSMS GC, university Lahore. This research is partially supported by Higher Education Commission of Pakistan.

\vskip 0,65 true cm

\medskip

\null\hfill  Abdus Salam School of Mathematical Sciences,\\
\null\hfill  GC University Lahore, Pakistan. \\
\null\hfill E-mail: {yameen99khan@gmail.com}

\end{document}